\documentclass[11pt]{amsart}

\usepackage[all]{xy}
\usepackage{amssymb,xspace}
\usepackage{graphicx}

\newcommand*{\C}{\ensuremath{\mathbb C}\xspace}

\newcommand*{\E}{\ensuremath{\mathcal E}\xspace}

\newcommand*{\Ooo}{\ensuremath{\mathcal O}\xspace}
\newcommand*{\PP}{\ensuremath{\mathbb P}\xspace}

\newcommand*{\Q}{\ensuremath{\mathbb Q}\xspace}

\newcommand*{\T}{\ensuremath{\mathcal T}\xspace}

\newcommand*{\Z}{\ensuremath{\mathbb Z}\xspace}

\DeclareMathOperator{\codim}{codim}

\DeclareMathOperator{\Num}{Num}

\DeclareMathOperator{\Spec}{Spec}

\newtheorem{proposition}{Proposition}[section]
\newtheorem{theorem}[proposition]{Theorem}

\theoremstyle{remark}
\newtheorem{note}[proposition]{Remark}

\theoremstyle{definition}
\newtheorem{definition}[proposition]{Definition}

\author{Serge Lvovski}
\address{National Research University Higher School of Economics,
AG Laboratory, HSE, 7 Vavilova str., Moscow, Russia, 117312}
\email{lvovski@gmail.com}

\title{On surfaces with zero vanishing cycles} \keywords{Vanishing
  cycles, monodromy} 
\subjclass{14D05, 14N99} 
\thanks{The author is
  partially supported by AG Laboratory NRU HSE, RF government grant,
  ag.  11.G34.31.0023}

\begin{document}

\begin{abstract}
We show that using an idea from a paper by Van de Ven one can obtain a
simple proof of Zak's classification of smooth projective surfaces
with zero vanishing cycles. This method of proof allows one to extend
Zak's theorem to the finite characteristic case.
\end{abstract}

\maketitle

\section*{Introduction}

In his paper~\cite{Zak}, Fyodor Zak obtained a complete classification
of smooth projective surfaces over \C for which ``Condition (A)'' from
Expos\'e~XVIII of SGA7 fails to be satisfied (see
Definition~\ref{cond_(A)} below). It is well known (see, for example,
\cite[Section~1]{Lanteri} or \cite[Proposition 6.1]{Lvovski}) that for
surfaces the violation of Condition~(A) is equivalent to triviality of
vanishing cycles or (in characteristic zero) to the emptiness of the
adjoint linear system.

Zak's elegant proof, being based ultimately on the theory of
degeneration of isolated singularities, does not appear to be directly
applicable to the case of finite characteristic.

The aim of this paper is to show that, using an idea from Van de
Ven's article~\cite{VandeVen}, one can produce a simple proof of Zak's
result that is valid over an arbitrary base field.

It would be interesting to learn something about higher-dimensional
varieties \emph{not} satisfying Condition~(A), at least in
characteristic zero to begin with.

Zak's result was reproved later (and independently: Antonio Lanteri
communicated to me in a letter that he and Palleschi were unaware of
Zak's paper while preparing their articles) by Lanteri and Palleschi
\cite[Proposition 3.1]{LanteriPalleschi}. Their proof also depends on
the $\mathrm{char}=0$ assumption. In the paper~\cite{Lanteri1980}
Lanteri uses a construction resembling our proof of Theorem~\ref{main}
to obtain a characterisation of projectively ruled surfaces, but the
condition the author imposes on the vanishing cycles
(see~\cite[Section~3, Condition~(T)]{Lanteri1980}) is much harder to
check than just vanishing.

\subsection*{Acknowledgements}
It is a pleasure to express my gratitude to Fyodor Zak, who introduced
me to this subject many years ago. Besides, Fyodor has carefully read
the first version of this text and made me aware of the related works
of Lanteri and Palleschi, as well as of the paper~\cite{Fary}. I am
grateful to Antonio Lanteri for sending me scans of his papers and for
useful email discussions. Last but not least, I would like to thank
Sergei Tabachnikov for providing me with a copy of the
paper~\cite{Rama}.

\section{The condition (A)}

Suppose that $X\subset\PP^r$ is a smooth projective variety of
dimension $n$ over an algebraically closed field and $Y\subset X$ is
its smooth hyperplane section. For the case in which the embedding
$X\hookrightarrow \PP^r$ is Lefschetz (in characteristic zero every
embedding is Lefschetz, in characteristic $p$ it means that $X$ is
reflexive or the dual variety $X^*\subset(\PP^r)^*$ has codimension
greater than $1$; see Remark~\ref{SGA_vs_modern} below), N.~Katz
introduced in SGA7 the following ``Condition (A)''.
\begin{definition}[N. Katz]\label{cond_(A)}
If $X\subset\PP^r$ is a Lefschetz embedding and $Y\subset X$ is a
smooth hyperplane section, then one says that condition~(A) is
satisfied for this embedding if either $\dim X^*<r-1$, where
$X^*\subset (\PP^r)^*$ is the dual variety, or $\dim X^*=r-1$ and the
homomorphism $i^*\colon H^{n-1}(X)\to H^{n-1}(Y)$ induced by the
embedding $i\colon Y\hookrightarrow X$ is \emph{not} an isomorphism.
\end{definition}
Here, $H^k(Z)$ means $H^k(Z,\Q)$ (singular cohomology) if the base
field is~\C, and $H^k(Z,\Q_\ell)$ ($\ell$-adic cohomology, where
$\ell$ is different from characteristic) in general. Lefschetz
hyperplane theorem asserts that $i^*$ is always injective; if
Condition~(A) is satisfied for at least one smooth hyperplane
section~$Y\subset X$, it is satisfied for any smooth hyperplane
section. Finally, Condition~(A) is always satisfied if $\dim X$ is odd
and $X$ is not a linear subspace. If $\codim X^*=1$, Condition~(A) is
equivalent to the assertion that vanishing cycles with respect to the
generic (equivalently: at least one) Lefschetz pencil corresponding to
the embedding $X\subset\PP^r$ are not zero. See \cite[Expos\'e~XVIII,
  \S\,5.3 and Theorem~6.3]{SGA7.2}.  Katz and Deligne write that
Condition~(A) has a strong tendency to hold (``cette condition (A) a
une nette tendence \`a \^etre v\'erifi\'ee''), so it is interesting to
describe the exceptional cases where it is not satisfied.

Denote by
\[
C(X)=\{(x,t)\in X\times(\PP^r)^*\colon\text{$H_t$ is tangent
to $X$ at $x$} \}
\]
the \emph{conormal variety} of $X$ ($H_t\subset\PP^r$ stands for
the hyperplane in $\PP^r$ corresponding to the point
$t\in(\PP^r)^*$ in the dual projective space).  Recall (see, for
example, \cite[Theorem 1.1]{Kaji}) that a projective variety is called
\emph{reflexive} if the projection $C(X)\to X^*\subset(\PP^r)^*$ is
separable (so, in characteristic zero everything is
reflexive).

\begin{note}\label{SGA_vs_modern}
The definition of Lefschetz embeddings in SGA7
(see~\cite[Expos\'e~XVII, Definitions 2.2 and 2.3]{SGA7.2}) is
equivalent to the following. An embedding $X\subset\PP^r$, where $X$
is smooth, is Lefschetz if for a general line $L\subset(\PP^r)^*$ the
following conditions are satisfied.

(a) The $(r-2)$-dimensional linear subspace ${}^\perp L\subset\PP^r$
corresponding to the line $L\subset(\PP^r)^*$, is transversal to~$X$.

(b) There exists a non-empty Zariski open subset $U\subset L$ such
that for each $t\in U$ the hyperplane $H_t$ is transversal to $X$.

(c) If $t_0\in L$ is such that $H_t$ is not transversal to $X$, then
the scheme $H_{t_0}\cap X$ is a reduced variety with only one singular
point, and this point is the ordinary quadratic singularity.

Here, condition~(a) follows from a simple version of Bertini theorem,
condition~(b) means that $L$ is not contained in $X^*$, and
condition (c) is satisfied if either $\dim X^*\le r-2$ or $\dim
X^*=r-1$ and the set of points $t\in X^*$ over which derivative of the
projection morphism $q\colon C(X)\to X$ has maximal rank everywhere
is non-empty. The latter condition is equivalent to the separability
of the morphism~$q$.
\end{note}

It is well known that if $X$ is a smooth reflexive surface and not a
linear subspace, then $X^*$ is a hypersurface. See, for example,
\cite[Proposition~1]{Zak}, where the proof is valid in arbitrary
characteristic provided that $X$ is reflexive, or Landman's ``parity
theorem'' \cite[Theorem II(21)]{Kleiman}. Thus, in the two-dimensional
case the condition ``$X\hookrightarrow\PP^r$ is a Lefschetz embedding'' is
equivalent to the reflexivity of $X$. Note also that in the definition
of Condition~(A) the ambient projective space can be safely replaced
by the linear span of $X$, so in the sequel we may and will assume
that $X$ is not contained in a hyperplane.

It is worth mentioning that the non-triviality of Condition~(A) had
already been observed in the pre-Grothendieck epoch. At least, this
``exceptional case'' is mentioned explicitly in the paper~\cite{Fary}
(see the note at the end of p.~37), where the author indicates that
the existence of embedded varieties for which Condition~(A) does not
hold had been known to J.~Leray.

\section{Two auxiliary results}

In this section we state two well-known folklore results about
projective surfaces. For the sake of completeness, we sketch the proofs.

\begin{definition}
Let us say that a smooth projective surface $X\subset\PP^r$ is
\emph{projectively ruled} if $X$ is swept by a $1$-dimensional family
of disjoint lines.
\end{definition}

\begin{proposition}\label{prop:ruled}
If a smooth projective surface $X\subset\PP^r$ contains a line $L$
with self-intersection index $(L,L)=0$, then $X$ is projectively
ruled. If $X$ is projectively ruled, then there exists a smooth
projective curve $C$ and a locally free sheaf \E on $C$ of rank $2$
such that $X\cong \PP(E)$ and $\Ooo_X(1)\cong \Ooo_{X|C}(1)$. 
\end{proposition}

\begin{proof}[Sketch of proof]
It is easy to see that the result of any flat deformation of a line
$L\subset X$ is again a line. If $C$ is the connected component of the
Hilbert scheme of lines on $X$ which (the component) contains the
point corresponding to $L$, then for any deformation $L'$ of
the line $L\subset X$ one has $(L',L')=0$, whence $h^0(N_{X|L'})=1$,
$h^i(N_{X|L'})=0$ for $i>0$, so $C$ is a smooth projective curve. If
$\pi\colon T\to C$ is the family of lines on~$X$ corresponding to $C$ and
$p\colon T\to X$ is the canonical projection, then it is easy to see
that $p$ is separable. Indeed, if $y\in T$ is a closed point and
$L=p(\pi^{-1}(y))\subset X$, then the
restriction of derivative of $p$ to the tangent space $\T_yT$ is
isomorphic onto its image. If $z=\pi(y)\in C$, then in the commutative
diagram
\[
\xymatrix{
{\T_yT/\T_y\pi^{-1}(z)}\ar[r]\ar[d]&{H^0(L,N_{X|L})}\\
{\T_zC}\ar[ur]
}
\]
both the vertical and the diagonal arrow are isomorphisms, whence the
horizontal arrow is also an isomorphism, so the mapping
$\T_yT\to\T_{p(y)}X$ is non-degenerate. Moreover, $p$ is generically
one to one since self-intersection index of each line in the family
is~$0$. Since $X$ is smooth, it follows that $p\colon T\to X$ is an
isomorphism. Identifying $X$ with $T$, it suffices to put
$\E=\pi_*\Ooo_X(1)$.
\end{proof}

\begin{proposition}\label{prop:plane}
Suppose that $X$ is a smooth projective surface and $C\subset X$ is a
curve such that $C\cong \PP^1$, the self-intersection index $(C,C)$
equals~$1$, and $C\subset X$ is an ample divisor. Then $X\cong\PP^2$.
\end{proposition}

\begin{proof}[Sketch of proof]
Any flat deformation $C'\subset X$ of the curve $C\subset X$ is
isomorphic to $\PP^1$. Indeed, $\chi(\Ooo_{C'})=1$ and $C'$ is
irreducible since $(C',C)=1$ and $C$ is ample. Hence,
$h^0(N_{X|C'})=2$, $h^i(N_{X|C'})=0$ for $i>0$, so if $B$ is the
connected component of the Hilbert scheme of curves on $X$ which (the
component) contains the point corresponding to $C$, then $B$ is a
(smooth) projective surface. If
\begin{equation*}
  \xymatrix{
    T\ar[d]_p\ar[r]^q &X\\
 B
  }    
\end{equation*}
is the standard diagram representing the family of curves on $X$
parametrized by~$B$, then, for a general (closed) point $x\in X$, one
has $\dim q^{-1}(x)=1$.

Let $\sigma\colon \tilde X\to X$ be the blowup of $X$ at~$x$. Proper
transforms (with respect to $\sigma$) of the curves from the
family~$B$ passing through~$x$, are isomorphic to $\PP^1$ and have
zero self-intersection. Arguing as in the proof of
Proposition~\ref{prop:ruled}, we conclude that $\tilde X$ admits a
morphism $\pi\colon \tilde X\to C$ onto a smooth curve $C$ such that
the fibers of~$\pi$ are the above-mentioned proper transforms, all
isomorphic to~$\PP^1$. Restricting $\pi$ to
the exceptional curve $E=\sigma^{-1}(x)\subset\tilde X$, one concludes
that there exists a surjective morphism $E\to C$, whence $C\cong
\PP^1$ by L\"uroth's theorem. Besides, $E$ is a section of the
morphism~$\pi$, so $\tilde X$ is a $\PP^1$-bundle
over $\PP^1\cong C$, so $X\cong\PP(\Ooo_{\PP^1}\oplus
\Ooo_{\PP^1}(d))$ for some $d\ge 0$. Since this $\PP^1$-bundle has a
section~$E$ with self-intersection equal to~$-1$, one concludes that
$d=1$; the blowdown of such a section of $\PP(\Ooo_{\PP^1}\oplus
\Ooo_{\PP^1}(1))$ is isomoprhic to $\PP^2$, and we are done.
\end{proof}

\section{Statement and proof}

The following theorem was first proved by Zak~\cite{Zak} over~\C.

\begin{theorem}\label{main}
Suppose that $X\subset\PP^r$ is a smooth reflexive surface not lying
in a hyperplane. Then $X$ fails to satisfy Condition~\textup{(A)} if
and only if one of the following conditions holds.

\textup{(i)} $X=\PP^2$.

\textup{(ii)} $X$ is projectively ruled.

\textup{(iii)} $X=v_2(\PP^2)\subset\PP^5$ \textup(the second Veronese image
of $\PP^2$\textup).

\textup{(iv)} $X\subset\PP^4$ is an isomorphic projection of the
surface $v_2(\PP^2)\subset\PP^5$.
\end{theorem}

Since it is clear that Condition~(A) is not satisfied for the plane
$\PP^2\subset\PP^r$, from now on we assume that $X$ is not a linear
subspace of $\PP^r$.

Proof of the theorem is based on the following observation.

\begin{proposition}\label{prop:reducible}
Suppose that $X\subset\PP^r$ is a smooth projective reflexive surface
and not a linear subspace of $\PP^r$. Then the following two
conditions are equivalent.

\textup{(i)} Condition~\textup{(A)} fails for the embedding
$X\hookrightarrow\PP^r$.

\textup{(ii)} There exists a hyperplane $H\subset\PP^r$ for which
$H\cap X$ is reduced, reducible, and smooth except for one ordinary
quadratic singularity.

Moreover, if \textup{(ii)} holds, then $H\cap X=Y_1\cup Y_2$, where
$Y_1$ and $Y_2$ are smooth irreducible curves intersecting
transversally at one point, and this is also the case for any hyperplane
$H'\subset\PP^r$ for which $H'\cap X$ is smooth except for one
ordinary quadratic singularity.
\end{proposition}

\begin{proof}[Proof of the proposition]
To begin with, recall that \emph{any} hyperplane section $Y\subset X$
is connected (see, for example,~\cite[Corollary~III.7.9]{Hartshorne}).
Choose a hyperplane $H\subset\PP^r$ for which $Y=H\cap X$ is smooth
except for one ordinary quadratic singularity. The hyperplane $H$ can
be included in a Lefschetz pencil $L\subset(\PP^r)^*$ (see
\cite[Expos\'e~XVII, Definition 2.2]{SGA7.2}; see also
\cite[\S\,1.6]{Lamotke} for the case of varieties over~\C). If $\tilde
X$ is the blow-up of $X$ at the finite set ${}^\perp L\cap X$ (where
${}^\perp L\subset\PP^r$ is the linear space of codimension $2$
corresponding to~$L\subset(\PP^r)^*$), then this Lefschetz pencil is a
morphism $\pi\colon\tilde X\to L$ such that for each (closed) point
$t\in L$, its fiber over $t$ is isomorphic to $X\cap H_t$. Recall some
basic facts from Picard--Lefschetz theory.

If the base field is \C and $Y_0=X\cap H_t$ has an ordinary quadratic
singularity, then for all $t'$ close enough to $t$ the intersection
$Y'=X\cap H_{t'}$ is smooth and contains an
embedded circle $c\subset Y'$ (``vanishing cycle'') such that $Y_0$ is
homeomorphic to $Y'/c$ and the class $\delta$ of $c$ in $H^1(Y',\Q)$
equals zero if and only if $b_1(X)=b_1(Y)$. Thus, Condition (A) fails
if and only if $c=0$; now it follows from the cohomology exact sequence
\begin{equation}\label{eq:classical}
H^1(Y',\Z)\to H^1(c,\Z)\to H^2(Y_0,\Z)\to H^2(Y',\Z)\to H^2(c,\Z),
\end{equation}
in which the leftmost arrow is zero since the class of $c$ is zero,
that $b_2(Y_0)=2$, whence $Y_0=Y_1\cup Y_2$ is union of two smooth
components intersecting transversally at one point. If, on the other
hand, Condition~(A) is satisfied, then the leftmost arrow
in~\eqref{eq:classical} is injective, whence $b_2(Y_0)=b_2(Y)=1$ and $Y_0$
is irreducible.

In arbitrary characteristic the same argument requires a slightly
different wording. As usual, $\ell$ will denote a prime different from
the characteristic; since the base field is algebraically closed, we
may and will identify $H^*(\cdot,\Q_\ell(j))$ with
$H^*(\cdot,\Q_\ell)$. If a point $t\in L$ is such that $X\cap H_t$ has
an ordinary quadratic singularity, put, according to SGA7,
$A=\widehat{\Ooo_{L,t}}$ (completion of the local ring) and $S=\Spec A$; by
$\bar\eta$ denote $\Spec$ of the algebraic closure of the field of
fractions of~$A$. If $Y_0=H_t\cap X$, $\hat\pi\colon\hat X\to S$ is
the pullback of the morphism $\pi\colon\tilde X\to S$ with respect to the
morphism $S\to L$, and $Y_{\bar\eta}$ is the general geometric fiber
of $\hat\pi$, then there exists a class $\delta\in
H^1(Y_{\bar\eta},\Q_\ell)$ (the vanishing cycle) and an exact sequence
\begin{equation}\label{exseq}
0\to H^1(Y_0,\Q_\ell)\to
H^1(Y_{\bar\eta},\Q_\ell)\xrightarrow{(\cdot,\delta)}
\Q_\ell\to H^2(Y_0,\Q_\ell)\to
H^2(Y_{\bar\eta},\Q_\ell)\to 0
\end{equation}
(see \cite[Expos\'e XV, Theorem 3.4]{SGA7.2} or
\cite[4.3.3]{Deligne1}). Now the following conditions are equivalent.

(a) $\delta=0$.

(b) The arrow $H^1(Y_{\bar\eta},\Q_\ell)\to\Q_\ell$ in~\eqref{exseq} is
zero.

(c) $b_2(Y_s)=b_2(Y_{\bar\eta})+1$.

(d) Condition~(A) fails for the embedding $X\subset\PP^r$.

Indeed, since $Y_{\bar\eta}$ is a smooth and connected projective
curve, Poincar\'e duality shows that $\mathrm{(a)}\Leftrightarrow
\mathrm{(b)}$, the equivalence $\mathrm{(a)}\Leftrightarrow
\mathrm{(c)}$ follows from~\eqref{exseq}, and the equivalence
$\mathrm{(a)}\Leftrightarrow \mathrm{(d)}$ follows from main results
of Picard--Lefschetz theory (cokernel of the injection
$H^1(X,\Q_\ell)\to H_1(X\cap H,\Q_\ell)$ is generated by ``the''
vanishing cycles and all the vanishing cycles are conjugate). Thus,
Condition~(A) fails if and only if $b_2(Y_0)=2$, so the curve $Y_0$
has two irreducible components; since the only singular point of this
curve is ordinary quadratic, these components intersect transversally
at one point. If, on the other hand, $\delta\ne0$, then the exact
sequence~\eqref{exseq} shows that $b_2(Y_s)=b_2(Y_{\bar\eta})=1$, so
$Y_s$ is irreducible.
\end{proof}

Now we pass to the proof of Theorem~\ref{main}. 

\begin{proof}[Proof of the ``if'' part of Theorem~\ref{main}]
We are to check that Condition~(A) fails for projectively ruled
surfaces, the Veronese surface, and its projection. If $X\subset\PP^r$ is
projectively ruled, $p\in X$, and $L$ is the line of the ruling
passing through $p$, then $L$ is contained in the embedded tangent
space $T_pX\subset\PP^r$, so each hyperplane $H$ that is tangent to
$X$ at $p$ must contain~$L$. Thus, if $H\cap L$ has only an ordinary
quadratic singularity at $p$, then the curve $H\cap X$ contains $L$,
and since $L$ has zero self-intersection,  $H\cap X$ must have other
components, and Proposition~\ref{prop:reducible} shows that
Condition~(A) fails.

If $X$ is the Veronese surface or its projection, observe that if
$D\in|\Ooo_{\PP^2}(2)|$ is a curve with one singularity, then $D$ is
union of two different lines, so $D$ is reducible and
Proposition~\ref{prop:reducible} completes the proof again.
\end{proof}

\begin{proof}[Proof of the ``only if'' part of Theorem~\ref{main}]
The main idea of this proof is borrowed from Van de Ven's
paper~\cite{VandeVen} (see proof of Theorem~I therein).

Suppose that Condition (A) fails for a smooth surface $X\subset\PP^r$,
where $X$ is not a linear space. Proposition~\ref{prop:reducible}
implies that $X$ has a hyperplane section of the form $Y_1+Y_2$, where
$Y_1$ and $Y_2$ are smooth, irreducible, and intersect transversally
at one point.  Since $Y_1+Y_2$ is a hyperplane section of~$X$, one
has, for $j=1$ or~$2$,
\begin{equation}\label{eq:pos}
(Y_j,Y_1+Y_2)=\deg Y_j>0;
\end{equation}
observing that $(Y_1,Y_2)=1$ (since $Y_1$ and $Y_2$ intersect
transversally at one point), one concludes from~\eqref{eq:pos} that
$(Y_j,Y_j)>-1$, so 
\begin{equation}\label{both_non-neg}
(Y_1,Y_1)\ge0,\quad (Y_2,Y_2)\ge0.
\end{equation}

Denote by $V\subset \Num(X)\otimes\Q$, where $\Num(X)$ is the group of
divisors on $X$ modulo numeric equivalence, the subspace generated by
the classes of $Y_1$ and $Y_2$. Since the curve
$Y_1+Y_2$ has positive self-intersection,
it follows from Hodge index theorem that only the following two cases
are possible:

(a) $\dim V=2$ and
\begin{equation}\label{eq:Hodge}
\begin{vmatrix}
(Y_1,Y_1)&(Y_1,Y_2)\\
(Y_1,Y_2)&(Y_2,Y_2)
\end{vmatrix}
<0;
\end{equation}

(b) $\dim V=1$ and classes of $Y_1$ and $Y_2$ are proportional.

In case~(a), inequality~\eqref{eq:Hodge} implies that
\[
(Y_1,Y_1)(Y_2,Y_2)<1,
\]
so it follows from \eqref{both_non-neg} that at least one of the
self-intersection indices $(Y_1,Y_1)$ or $(Y_2,Y_2)$ must be zero. If
one of them (say, $(Y_1,Y_1)$) equals zero and the other equals~$1$,
then $\deg Y_1=(Y_1,Y_1+Y_2)=1$, so $Y_1$ is a line with
self-intersection index~$0$, and Proposition~\ref{prop:ruled} shows
that $X$ is projectively ruled. If they both are zero, then $\deg
X=(Y_1+Y_2,Y_1+Y_2)=2$, so $X$ is a quadric, and the smooth
two-dimensional quadric is projectively ruled.

In case~(b), $Y_2$ is numerically equivalent to $rY_1$, where $r$ must
be positive since both $Y_1$ and $Y_2$ are effective divisors. Since
\begin{equation}\label{eq:Y^2}
1=(Y_1,Y_2)=r(Y_1,Y_1)=r^{-1}(Y_2,Y_2)
\end{equation}
and both $(Y_1,Y_1)$ and $(Y_2,Y_2)$ are integers, it follows that
$r=(Y_1,Y_1)=(Y_2,Y_2)=1$ and $Y_1\approx Y_2$, where $\approx$ means
numeric equivalence. Since $Y_1+Y_2$ is an ample divisor, it follows
that $Y_1$ is ample. Now \eqref{eq:Y^2} implies that $\deg
Y_1=(Y_1,Y_1)+(Y_1,Y_2)=2$, so $Y_1$ is a conic, whence
$Y\cong\PP^1$. Proposition~\ref{prop:plane} implies that $X$ is
isomorphic to $\PP^2$, and this isomorphism takes $Y_1$ and $Y_2$ to
lines since self-intersection indices of these curves equal~$1$. Thus,
$(X,\Ooo_X(1))\cong (\PP^2,\Ooo_{\PP^2}(2))$, whence $X$ is projectively
isomorphic to $v_2(\PP^2)\subset\PP^5$ or to an isomorphic projection
of this surface. It is well known that secant variety of $v_2(\PP^2)$
has dimension~$4$, so an isomorphic projection of $v_2(\PP^2)$ must
lie in $\PP^4$. This completes the proof.
\end{proof}

Observe finally that it follows from Theorem~\ref{main} that if $\dim X=2$
and $X$ is not a linear subspace, one can put $d_0=2$ in Corollary~6.4
from~\cite[Expos\'e XVIII]{SGA7.2}. Indeed, if $d\ge 2$ and $X$ is a
surface satisfying one of the conditions~(ii)--(iv) of the theorem,
then $v_d(X)$ does not satisfy any of them.

\bibliographystyle{amsalpha}

\bibliography{monodromy}

\end{document}